\documentclass[nopagetitle]{article}
\usepackage[utf8]{inputenc}
\usepackage{amssymb}
\usepackage{amsmath}
\usepackage{amsthm}
\usepackage{authblk}
\usepackage{xcolor}
\usepackage{enumitem}

\newtheorem{theorem}{Theorem}[section]
\newtheorem{proposition}[theorem]{Proposition}
\newtheorem{lemma}[theorem]{Lemma}
\newtheorem{definition}[theorem]{Definition}

\newtheorem{corollary}[theorem]{Corollary}

\begin{document}
\everymath{\displaystyle}

\title{Tame pseudofinite theories with wild pseudofinite dimensions}
\author{Alexander Van Abel}

\maketitle

\begin{abstract}
We construct two pseudofinite theories which are tame from a neostability perspective, yet have pathological fine pseudofinite dimension in all models. These theories serve as counterexamples to potential converses of results by Garcia, Macpherson and Steinhorn relating pseudofinite dimension to tameness. We demonstrate that pseudofinite cardinality in these theories is well behaved with regards to definability, and provide a novel method of proving quantifier elimination using pseudofinite cardinality.
\end{abstract}

\section{Introduction}

In their paper ``Pseudofinite structures and simplicity'' \cite{psas}, authors Dario Garcia, Dugald Macpherson and Charles Steinhorn proved a number of results relating stability-theoretic notions in an infinite ultraproduct of finite structures to conditions on the dimension operator $\delta = \delta_{fin}$, which is the ``fine pseudofinite dimension'' introduced by Hrushovski in his paper \cite{hrush}.

One way to define $\delta$ is as follows. Let $(M_i : i \in I)$ be a family of finite $L$-structures. Let $M = \prod_{i \to \mathcal{U}} M_i$ be an ultraproduct. Take two formulas $\varphi(\bar{x},\bar{y})$ and $\psi(\bar{x},\bar{y})$, and tuples $\bar{a}$ and $\bar{b}$ of the same length as $\bar{y}$. We say \[\delta(\varphi(M,\bar{a})) = \delta(\psi(M,\bar{b}))\] if there is some fixed natural number $n$ such that for almost all $i$ we have $\frac{1}{n} \leq \frac{|\varphi(M_i, \bar{a}_i)|}{|\varphi(M_i, \bar{b}_i)|} \leq n$, and we say \[\delta(\varphi(M, \bar{a})) < \delta(\psi(M,\bar{b}))\] if for every natural number $n \in \mathbb{N}$, we have $n \cdot |\varphi(M_i, \bar{a}_i)| < |\psi(M_i, \bar{b}_i)|$ for almost all $i$. In this sense, $\delta$ is comparing asymptotic growth rates of the sizes of the defined subsets in the finite structures $M_i$.

In \cite{psas}, the authors specify the following conditions on the sequence $(M_i : i \in I)$, which we call \emph{GMS conditions} in this paper. Let $\varphi(\bar{x},\bar{y})$ be a formula. Then GMS condition $(A)_\varphi$ says the following: there is no sequence of parameters $\bar{a}_1, \bar{a}_2, \ldots$ such that \[\delta\big{(}\varphi(M, \bar{a}_1)\big{)} > \delta\big{(}\varphi(M, \bar{a}_1) \wedge \varphi(M, \bar{a}_2)\big{)} > \delta\big{(} \varphi(M, \bar{a}_1) \wedge \varphi(M, \bar{a}_2) \wedge \varphi(M,\bar{a}_3)\big{)} > \ldots.\] The authors of that paper then show that if the sequence $(M_i : i \in I)$ satisfies $(A)_\varphi$ for every formula $\varphi$ (which is what it means to satisfy GMS condition (A)) then the ultraproduct $M$ has a simple and low theory \cite[Theorem 3.2.2]{psas}.

The authors demonstrate that this implication does not reverse, in the following way (\cite[Example 4.1.3]{psas}). Take $T$ to be the theory of an equivalence relation with infinitely many infinite classes. This is a stable, hence simple, theory. It is also pseudofinite: it is the theory of any ultraproduct $\prod_{i \to \mathcal{U}} M_i$ of finite equivalence structures such that both the number of equivalence classes and the size of the smallest class tend to infinity with respect to the ultrafilter. The authors construct such a sequence of structures $(M_n : n \in \omega)$ in a way that violates the GMS condition $(A)_\varphi$ for the formula $\varphi(x,y) := $ ``$\neg x E y$''. For $n \in \omega$, let $M_n$ have $n$ equivalence classes, one of size $n^i$ for $i = 1, \ldots, n$ (so that $|M_n| = \sum_{i=1}^n n^i$). For $1 \leq k \leq n$ let $c^k_n$ be an element of the $k$th largest class in $M_n$, the one of size $n^{n-(k-1)}$. Then $\big{|}\{x \in M_n : (\neg x E c^1_n) \wedge \ldots \wedge (\neg x E c^k_n\} \big{|}$ is $\sum_{i=1}^{n-k} n^i$, which is approximately $\frac{1}{n^k} |M_n|$. This implies that in the ultraproduct, we have \[\delta\big{(} \varphi(M, c^1) \wedge \ldots \wedge \varphi(M, c^k) \big{)} > \delta\big{(} \varphi(M, c^1) \wedge \ldots \wedge \varphi(M, c^k) \wedge \varphi(M, c^{k+1})\big{)}\] for each $k$, where $c^k$ is the element of the ultraproduct represented by the sequence $(c^k_1, c^k_2, c^k_3,\ldots)$. Therefore, GMS condition $(A)_\varphi$ fails, even though the theory is simple and low (superstable, even).

However, we can find a model of $T$ which is an ultraproduct of finite equivalence structures $M'_n$ where GMS condition $(A)_\varphi$ \emph{is} satisfied for every formula $\varphi$. We can take $M'_n$ to have $n$ equivalence classes, each of size $n$. Then the $\delta$ dimension is highly well-behaved -- in particular $(A)_\varphi$ is satisfied for every $\varphi$ (giving a rather roundabout proof that $T$ is a simple theory). 

This led the author of this document to the following question. Let $T$ be a pseudofinite simple and low theory (which includes the class of stable theories). Since $T$ is pseudofinite, it may be modelled by an ultraproduct of finite structures $\prod_{i \to \mathcal{U}} M_i$. As shown above, it may be the case that the sequence $(M_i : i \in I)$ does not satisfy GMS condition $(A)_\varphi$ for every $\varphi$. But will it be the case that there is \emph{some} sequence $(M_i : i \in I)$ of finite structures and ultrafilter $\mathcal{U}$ such that $\prod_{i \to \mathcal{U}} M_i \models T$ and $(M_i : i \in I)$ satisfies $(A)_\varphi$ for all $\varphi$? We can ask a similar question for supersimple $T$ and the GMS condition (SA), to be introduced in the next section, inspired by the result in \cite{psas}  that if an infinite ultraproduct of finite structures satisfies (SA) then the theory of the ultraproduct is supersimple \cite[Theorem 3.3]{psas}.

In this paper, we give negative answers to both of these questions. In the section ``Supersimple does not imply (SA)'' we exhibit a pseudofinite theory which is supersimple of $U$-rank 1 such that no infinite ultraproduct of finite structures which models $T$ will satisfy GMS condition (SA), and in the section ``Simple does not imply (SA)'', we exhibit a pseudofinite theory which is simple and low (stable, in fact) such that no ultraproduct of finite structures satisfying $T$ satisfies GMS condition (A).

The author thanks Alf Dolich, Alice Medvedev, Charles Steinhorn, Dario Garcia, Alex Kruckman, and Cameron Donnay Hill for their comments and suggestions.

\color{black}

\section{Counting pairs and dimension conditions}

The way that we formalize pseudofinite dimension $\delta$ is through the notion of \emph{pseudofinite cardinality}, which is itself formalized by what Garcia terms a \emph{counting pair}.

For a (first-order, single-sorted) language $L$, let $L^+$ be the two-sorted expansion $(L,OF)$, where $OF = (0,1,+,-,\cdot,<)$ is the language of ordered fields. Symbols from $L$ apply to the home sort, and symbols from $OF$ apply to the new second sort. Additionally, for every partitioned formula $\varphi(\bar{x},\bar{y})$ of $L$, we add a cross-sorted function $f_{\varphi(\bar{x},\bar{y})}$ which maps $|\bar{y}|$-tuples from the home sort $L$ to elements of the $OF$ sort.

Let $M_0$ be a finite $L$-structure. We extend $M_0$ to an $L^+$-structure $M_0^+$ by letting the $OF$ sort contain a copy of the ordered field $\mathbb{R}$, and for each partitioned formula $\varphi(\bar{x},\bar{y})$ and each $\bar{b} \in M_0^{|\bar{y}|}$ letting $f_{\varphi(\bar{x},\bar{y})}(\bar{b})$ be $|\varphi(M_0^{|\bar{x}|},\bar{b})| \in \mathbb{N} \subseteq \mathbb{R}$.

If $M$ is an ultraproduct of finite structures $\prod_{i \to \mathcal{U}} M_i$, then we expand $M$ to an $L^+$-structure $M^+$ by expanding each $M_i$ to $M_i^+$, and then taking $M^+$ to be the two-sorted ultraproduct $\prod_{i \to \mathcal{U}} M_i^+$. 

We denote the OF-sort of $M$ by $\mathbb{R}^\star$. The structure in this sort is an ultrapower of the ordered field of real numbers. When $X \subseteq M$ is definable, say by $\varphi(\bar{x},\bar{b})$ we let $|X|$ denote the value $f_\varphi(\bar{b}) \in \mathbb{R}^\star$. We note that this is a well-defined notation, in that if $\varphi(M^{|\bar{x}|},\bar{b}) = X = \varphi'(M^{|\bar{x}|},\bar{c})$ then this equality holds in almost all $M_i$, whence $f_\varphi(\bar{b}) = f_{\varphi'}(\bar{c})$). We denote the set of nonstandard integers $\{(r_i)_{i \to \mathcal{U}} \in \mathbb{R}^\star : r_i \in \mathbb{Z}$ for $\mathcal{U}$-almost many $i\}$ by $\mathbb{Z}^\star$, and note that for all definable sets $X$ we have $|X| \in \mathbb{Z}^\star$.

We note that the expansion $M^+$ does not depend purely on the structure $M$, but rather on the sequence of structures $M_i$ and the ultrafilter $\mathcal{U}$. That is, if $M = \prod_{i \to \mathcal{U}} M_i$ is isomorphic to $M' = \prod_{i \to \mathcal{U}'} M_i'$, it may not be the case that $M^+$ is isomorphic to $(M')^+$ -- as in the example given in the introductory section.

In this framework of pseudofinite cardinality, we can simplfy the definition of $\delta$ to $\delta(X) = \delta(Y)$ if there is some $n \in \mathbb{N}$ with $\frac{1}{n}|Y| \leq |X| \leq \frac{1}{n}|Y|$, and $\delta(X) < \delta(Y)$ if for all $n \in \mathbb{N}$ we have $n |X| < |Y|$, for definable sets $X$ and $Y$. Alternatively, $\delta(X) < \delta(Y)$ if and only if $\frac{|X|}{|Y|}$ is an infinitesimal element of $\mathbb{R}^\star$.

\begin{definition}
Let $M := \prod_{i \to \mathcal{U}} M_i$ be a pseudofinite ultraproduct, and let $M^+$ be its counting pair expansion. Let $\varphi(\bar{x},\bar{y})$ be an $L$-formula. We say that $M^+$ satifies the GMS condition $(A)_\varphi$ if there is \emph{no} sequence of $|\bar{y}|$-tuples $\bar{a}_1,\bar{a}_2,\ldots \in M$ such that $\delta(\varphi(\bar{x},\bar{a}_1) > \delta(\varphi(\bar{x},\bar{a}_1) \wedge \varphi(\bar{x},\bar{a}_2)) > \ldots$. We say that $M^+$ satisfies the GMS condition $(A)$ if it satisfies $(A)_\varphi$ for all $L$-formulas $\varphi(\bar{x},\bar{y})$.

We say $M^+$ satisfies the GMS condition (SA) if there is no sequence of definable sets $M \supseteq X_1 \supset X_2 \supset X_3 \supset \ldots$ such that $\delta(X_i) > \delta(X_{i+1})$ for each $i$.
\end{definition}

We usually abuse notation and say that $M$ satisfies (A) or (SA) rather than $M^+$, while keeping in mind that these properties depend on the particular ultrafilter and family of finite structures such that $M = \prod_{i \to \mathcal{U}} M_i$.

In \cite{psas}, the authors prove that if $M$ satisfies (A) then $Th(M)$ is simple, and if $M$ satisfies (SA) then $Th(M)$ is supersimple.

Pseudofinite cardinality inherits first-order expressible properties from finite cardinality. In particular, we will use the following fact about cardinalities of sets with constant-sized fibers, the proof of which follows directly from the finite case:

\begin{lemma}
\label{constfiber}
Let $X \subseteq M^n$ be a definable set in the pseudofinite structure $M$. Suppose there is $0 < m < n$ and a hyperreal $c \in \mathbb{R}^\star$ such that for all tuples $\bar{a}$ of length $m$, if $\{\bar{b} : (\bar{a}\bar{b}) \in X\}$ is nonempty then $|\{\bar{b} : (\bar{a}\bar{b}) \in X\}| = c$. Then $|X| = c \cdot |\{\bar{a} : \exists \bar{y} (\bar{a} \bar{y}) \in X\}|$.
\end{lemma}

The author is generally interested in the model theory of counting pairs for models of a given theory.

Let $L, L'$ be disjoint first-order languages and let $M$ be an $L$-structure and $N$ be an $L'$ structure. The \emph{disjoint union} $M \sqcup N$ is the two-sorted structure in the language $L \cup L'$, where the first sort is the $L$-structure $M$ and the second sort is the $L'$-structure $N$, with no additional relations defined. Definable subsets in $M \sqcup N$ are Boolean combinations of sets $X \times Y$, where $X \subseteq M$ and $Y \subseteq N$ are definable in $L$ and $L'$, respectively. Loosely speaking, $M \sqcup N$ is the ``minimal''  simultaneous two-sorted expansion of $M$ and $N$. In particular, every relation strictly on $M$ definable in $M \sqcup N$ is already definable by an $L$-formula in $M$, and similarly for $N$.

Let $\varphi(\bar{x},\bar{y})$ be a partitioned $L$-formula. Let $(\psi_1(\bar{y}),\ldots,\psi_m(\bar{y}))$ be a tuple of $L$-formulas (possibly with parameters) and let $(c_1,\ldots,c_m)$ be a tuple of hyperreals. We say that $(c_1,\ldots,c_m)$ and $(\psi_1(\bar{y}),\ldots,\psi_m(\bar{y}))$ \emph{give and define the cardinalities of } $\varphi(\bar{x},\bar{y})$ if
\begin{itemize}
\item For each $\bar{b} \in M^{|\bar{y}|}$, there is an $i \in \{1,\ldots,m\}$ such that $|\varphi(M^{|\bar{y}|},\bar{b})| = c_i$, and
\item For each $i$, $\psi_i(M^{|\bar{y}|}) = \{\bar{b} : |\varphi(M^{|\bar{y}|},\bar{b})| = c_i\}$.
\end{itemize}

\begin{proposition}
\label{defunion}
Let $M = \prod_{i \to \mathcal{U}} M_i$ be a pseudofinite ultraproduct. Let $A \subseteq M$ and $B \subseteq \mathbb{R}^\star = \prod_{i \to \mathcal{U}} \mathbb{R}$. Then the structure $M^+$ is interdefinable with the disjoint union $M \sqcup  \mathbb{R}^\star$ over $A \cup B$ if and only if for every formula $\varphi(x,\bar{y})$ with $x$ a single variable, there are formulas $\psi_1(\bar{y}),\ldots,\psi_n(\bar{y})$ with parameters from $A$ and hyperreals $c_1,\ldots,c_n \in \mathbb{R}^\star$ algebraic over $B$ which give and define the cardinalities of $\varphi(x,\bar{y})$.
\end{proposition}
\begin{proof}
The disjoint union $M \sqcup \mathbb{R}^\star$ is definable in any multi-sorted structure containing $M$ and $\mathbb{R}^\star$ as sorts. We show that the condition in the proposition is equivalent to $M^+$ being definable in the disjoint union.

Suppose $M^+$ is definable over $A \cup B$, and let $\varphi(x,\bar{y})$ be an $L$-formula. By assumption, the relation ``$f_{\varphi(x,\bar{y})}(\bar{b}) = c$'' as a property of $(\bar{b},c)$ is definable in $M \sqcup \mathbb{R}^\star$ over $A \cup B$. Every definable subset of $M \sqcup \mathbb{R}^\star$ is a Boolean combination of sets $X \times Y$, where $X$ is a definable subset of $M$ and $Y$ is a definable subset of $\mathbb{R}^\star$. Putting this Boolean combination into disjunctive normal form, we obtain $L(A)$-formulas $\psi_1(\bar{y}),\ldots,\psi_n(\bar{y})$ and $OF(B)$-formulas $\theta_1(t),\ldots,\theta_n(t)$ such that for all $\bar{b} \in M^{|\bar{y}|}$ and $c \in \mathbb{R}^\star$, \[f_{\varphi(x,\bar{y})}(\bar{b}) = c \mbox{ if and only if } M \sqcup \mathbb{R}^\star \models \bigvee_{i=1}^n [\psi_i(\bar{b}) \wedge \theta_i(c)].\] We may assume each $\psi_i(M^{|\bar{y}|})$ is nonempty. It then follows that each $\psi_i(\mathbb{R}^\star)$ is a singleton. Let $c_i$ be this unique element. Then for all $\bar{b} \in M^{|\bar{y}|}$, the pseudofinite cardinality $|\varphi(M,\bar{b})| = f_{\varphi(x,\bar{y})}(\bar{b})$ is $c_i$ for some $i$. Since $\psi_i$ uniquely defines $c_i$ over $B$, we obtain that $c_i$ is algebraic over $B$. If $c_i = c_j$ for some $i \neq j$, we replace $\psi_i(\bar{y})$ with $\psi_i(\bar{y}) \vee \psi_j(\bar{y})$ and remove $\psi_j, \theta_j$; in this way, we may assume the hyperreals $c_1,\ldots,c_n$ are disjoint. Then it follows that for each $i$, the set $\{\bar{b} \in M^{|\bar{y}|} : |\varphi(M,\bar{b})| = c_i\}$ is defined by $\psi_i(\bar{y})$, proving the forward direction of the proposition.

Now suppose that for every formula $\varphi(x,\bar{y})$ with $x$ a single variable, there are formulas $\psi_1(\bar{y}),\ldots,\psi_n(\bar{y})$ with parameters from $A$ and hyperreals $c_1,\ldots,c_n \in \mathbb{R}^\star$ algebraic over $B$ which give and define the cardinalities of $\varphi(x,\bar{y})$. We will show that we can remove the restriction that $x$ is a single variable, by induction on $|\bar{x}|$. The case $|\bar{x}| = 1$ is true by assumption. Suppose the statement is true of $|\bar{x}|$ and consider the partitioned formula $\varphi(w \bar{x}, \bar{y})$. Applying the single-variable case to the repartitioned formula $\varphi(w,\bar{x}\bar{y})$, we obtain formulas $\psi_1(\bar{x}\bar{y}), \ldots, \psi_m(\bar{x}\bar{y})$ and hyperreals $c_1,\ldots,c_m \in \mathbb{R}^\star$ which give and define $\varphi(w,\bar{x}\bar{y})$. Then by the inductive hypothesis, there are formulas $\theta_{i,j}(\bar{y})$ and cardinalities $c_{i,j} \in \mathbb{R}^\star$ for $1 \leq i \leq m$ and $1 \leq j \leq m_i$ which give and define $\psi_i(\bar{x},\bar{y})$.

For a tuple $\sigma = (\sigma_1,\ldots,\sigma_n) \in \prod_{i=1}^m [1,\ldots,m_i]$ let $\theta_\sigma(\bar{y})$ be the formula $\bigwedge_i \theta_{i, \sigma_i}(\bar{y})$. As the formulas $\theta_{i,1}(\bar{y}),\ldots,\theta_{i,m_i}(\bar{y})$ partition $M^{|\bar{y}|}$ for each $i$, the formulas $\theta_{\sigma}(\bar{y})$ partition $M^{|\bar{y}|}$ as $\sigma$ ranges over $\prod_{i=1}^m [1,\ldots,m_i]$, although some of the formulas may not be realized in $M$. If $M \models \theta_\sigma(\bar{b})$, then $|\varphi(M^{|\bar{x}| + 1},\bar{b})| = \sum_{i=1}^m |\{(c,\bar{a}) : M \models \varphi(c,\bar{a},\bar{b})$ and $M \models \psi_i(\bar{a})\}|$. By Lemma \ref{constfiber}, this is equal to $\sum_{i=1}^m c_i \cdot |\psi_i(M^{|\bar{x}|},\bar{b})|$, which is $c_\sigma := \sum_{i=1}^m c_i  c_{i, \sigma_i}$ as $\bar{b} \in \theta_{i,\sigma_i}(M^{|\bar{y}|})$. Identifying tuples $\sigma$ and $\sigma'$ such that $c_\sigma = c_{\sigma'}$ (and combining $\theta_\sigma$ and $\theta_{\sigma'}$ into the disjunction $\theta_\sigma \vee \theta_{\sigma'}$, and re-indexing the tuples with numbers, we obtain formulas $\theta_1(\bar{y}),\ldots,\theta_n(\bar{y})$ and hyperreals $d_1,\ldots,d_n \in \mathbb{R}^\star$ which give and define the cardinalities $\varphi(M^{|\bar{x}|+1},\bar{y})$. 

Therefore for every formula $\varphi(\bar{x},\bar{y})$, there are formulas $\psi_1(\bar{y}),\ldots,\psi_n(\bar{y})$ with parameters from $A$ and hyperreals $c_1,\ldots,c_n \in \mathbb{R}^\star$ algebraic over $B$ which give and define the cardinalities of $\varphi(\bar{x},\bar{y})$. Then the relation  ``$f_{\varphi(\bar{x},\bar{y})}(\bar{y}) = t$'' in $M^+$ is definable in $M \sqcup \mathbb{R}^\star$, by the formula $\bigvee_{i=1}^n \psi_i(\bar{y}) \wedge t = c_i$. Parameters in the formulas $\psi_i$ come from $A$, and the hyperreals $c_i$ are algebraic over $B$, proving the backwards direction of the proposition.
\end{proof}

\section{Supersimple does not imply (SA)}

In this section, we give a supersimple pseudofinite theory $T$ such that every infinite ultraproduct of finite models which satisfies $T$ fails to satisfy (SA).

Our language $L$ contains

\begin{itemize}
\item A unary predicate $U_\sigma$ for every string $\sigma \in \omega^{<\omega}$

\item A binary relation $B_{\sigma,\tau}$ for every pair of strings $\sigma,\tau \in \omega^{<\omega}$ of the same length.
\end{itemize}

Our theory $T$ contains the following axiom (a) and axiom schemata (b)-(h):

\begin{enumerate}[label=(\alph*)]
\item $U_\emptyset$ is the universe;

\item $U_\sigma \neq \emptyset$, for each string $\sigma$;

\item $U_\sigma \supset U_\tau$, for strings $\sigma$ and $\tau$ such that $\sigma$ is an initial segment of $\tau$;

\item $U_{\sigma i} \cap U_{\sigma j} = \emptyset$, for each string $\sigma$ and distinct numbers $i,j$;

\item $B_{\sigma, \tau}$ is (the graph of) a bijection from $U_{\sigma}$ to $U_{\tau}$, for strings $\sigma$ and $\tau$ of the same length;

\item $B_{\sigma,\tau}(x,y)$ if and only if $B_{\tau,\sigma}(y,x)$, for strings $\sigma$ and $\tau$ of the same length;

\item If $B_{\sigma,\tau}(x,y)$ and $B_{\tau,\rho}(y,z)$ then $B_{\sigma,\rho}(x,z)$, for strings $\sigma, \tau$ and $\rho$ of the same length;

\item If $U_{\sigma i}(x)$ and $U_\tau(y)$, then $B_{\sigma, \tau}(x,y)$ if and only if $U_{\tau i}(y)$ and $B_{\sigma i, \tau i}(x,y)$, for strings $\sigma$, $\tau$ of the same length and each number $i$.
\end{enumerate}

\begin{proposition}

$T$ has quantifier elimination.

\end{proposition}

\begin{proof}

It suffices to show that the formula $\exists \Phi x (x,\bar{y})$ is equivalent to a quantifier-free formula $\pi(\bar{y})$, where $\Phi$ is a conjunction of literals (atomic formulas and their negations). We may assume that each literal contains the variable $x$, and we may also assume that each variable $y_i$ in $\bar{y}$ appears exactly once, as the general case will follow this (for example, if $\exists x B_{0,1}(x,y_1) \wedge B_{0,2}(x,y_2)$ is equivalent to $\pi(y_1,y_2)$, then $\exists x B_{0,1}(x,y) \wedge B_{0,2}(x,y)$ is equivalent to $\pi(y,y)$). By the axioms of $T$, such a quantifier-free formula is equivalent up to re-arrangement of variables to $U(x) \wedge B(x,\bar{y})$, where $U(x)$ is \[U_\sigma(x) \wedge \neg U_{\sigma \tau_1}(x) \wedge \ldots \wedge \neg U_{\sigma \tau_n}(x),\] and $B(x,\bar{y})$ is \[B_{\sigma, \rho_1}(x,y_1) \wedge \ldots \wedge B_{\sigma, \rho_n}(x,y_m) \wedge \neg B_{\sigma, \theta_1}(x,y_{m+1}) \wedge \ldots \wedge \neg B_{\sigma, \theta_k}(x,y_{m+k}),\] where $\sigma, \tau_i, \rho_i$ and $\theta_i$ are strings, with $\sigma, \rho_i$ and $\theta_i$ all of the same length (and possibly with $n, m$ or $k$ equal 0, in which case the associated conjunction is empty). Then for a tuple $\bar{b} = (b_1,\ldots,b_{m+k})$ of elements, not necessarily distinct, we have $M \models \exists x [U(x) \wedge B(x,\bar{b})]$ if and only if the following three sets of conditions hold of the tuple $\bar{b}$:

\begin{itemize}
\item $B_{\rho_i, \rho_j}(b_i, b_j)$ for $i,j$ distinct numbers in $\{1,\ldots,m\}$

\item $\neg B_{\rho_i, \theta_j}(b_i,b_{m+j})$ for $i$ in $\{1,\ldots,m\}$ and $j$ in $\{1,\ldots,k\}$, and

\item $\neg U_{\rho_i \tau_j}(b_j)$ for $i$ in $\{1,\ldots,n\}$ and $j$ in $\{1,\ldots,m\}$.
\end{itemize}

These conditions are all together given by a single quantifier-free formula in $y_1,\ldots,y_{m+k}$, proving quantifier elimination.

\end{proof}

Let $M$ be a model of $T$, and $A \subset M$. By quantifier elimination, every type $p(x) \in S_1(A)$ is determined by the sets \[U(p) = \{\sigma \in \omega^{<\omega} : U_\sigma(x) \in p\},\] and \[B(p) = \{(a,\tau) \in M \times \omega^{<\omega} : B_{\sigma', \tau}(x,a) \in p \mbox{ for some } \sigma' \in U(p) \}.\] $p$ is an algebraic type if and only if $B(p)$ is nonempty. It follows that a non-algebraic type is determined by its restriction in $S(\emptyset$).

\begin{proposition}
\label{weakmin}
Let $M$ be a model of $T$. Let $A \subseteq B \subseteq M$ and let $p \in S(A)$ be a non-algebraic type. Then $p$ has a unique extension to a non-algebraic type $q \in S(B)$.
\end{proposition}
\begin{proof}
By quantifier-elimination, a type $p \in S(A)$ is determined by which atomic formulas or their negations are in it. $p$ is non-algebraic if and only if the formula ``$\neg B_{\sigma, \tau}(x,a)$'' is in $p$ for every $a \in A$ and every pair of strings $\sigma, \tau$ of the same length. 
\end{proof}

\begin{corollary}
$T$ is superstable of $U$-rank 1.
\end{corollary}
\begin{proof}
Superstability can be seen from type-counting, as it follows from Proposition \ref{weakmin} that $|S(A)| = |S(\emptyset)| + |A|$ for every set $A$. Therefore every type has a $U$-rank, and the fact that each non-algebraic type has a unique non-algebraic extension implies that each non-algebraic type has $U$-rank 1.
\end{proof}

\begin{proposition}
$T$ is pseudofinite.
\end{proposition}
\begin{proof}
We construct a sequence of finite structures $M_1, M_2, M_3, \ldots$ such that each axiom of $T$ is satisfied in $M_n$ for sufficiently large $n$.

Let $M_n$ have as its universe the set of all strings in $\{0,\ldots,n-1\}^n$. To avoid clashing notation with the language $L$, we denote elements of this universe by letter variables $a = a_1 a_2 \ldots a_n$. For $\sigma \in \{0,\ldots,n-1\}^{\leq n}$, we let $a \in U_\sigma(M_n)$ when $\sigma$ is an initial segment of $a$. Let $\sigma, \tau \in \{0,\ldots,n-1\}^{\leq n}$ have the same length. Suppose that $a \in U_\sigma(M_n)$ and $b \in U_\sigma(M_n)$, so that $a = \sigma a'$ and $b = \sigma b'$ for strings $a', b' \in \{0,\ldots,n-1\}^{\leq n}$. Then we set $B_{\sigma, \tau}(a,b)$ if and only if $a' = b'$.

If $\sigma \in \omega^{<\omega} / \{0,\ldots,n-1\}^{\leq n}$, we let $U_\sigma, B_{\sigma, \tau}$ and $B_{\tau, \sigma}$ be empty relations.

One easily verifies that the axioms for $T$ hold in $M_n$ whenever the strings $\sigma, \tau, \rho$ appearing in the axiom are elements of $\{0,1,\ldots,n-1\}^{\leq n}$ (or in the case of axiom schema (h), elements of $\{0,1,\ldots,n-1\}^{<n}$) and the numbers $i,j$ are elements of $\{0,1,\ldots,n-1\}$. Therefore each axiom holds in $M_n$ for sufficiently large $n$.

\end{proof}

\begin{proposition}
Let $M$ be a an ultraproduct of finite structures which satisfies $T$. Let $f \in \omega^{\omega}$. For each $n$ let $\sigma_n$ be the initial segment of $f$ consisting of the first $n$ entries. Then $\delta(M) > \delta(U_{\sigma_1}(M)) > \delta(U_{\sigma_2}(M)) > \ldots$. In particular, $M$ fails to satisfy condition (SA).
\end{proposition}
\begin{proof}

We show that for every string $\sigma \in \omega^{<\omega}$ and every number $k \in \omega$, $\delta(U_\sigma(M)) > \delta(U_{\sigma k}(M))$, which proves the proposition. For all numbers $j \in \omega$, the relation $B_{\sigma k, \sigma j}$ gives a bijection between $U_{\sigma k}(M)$ and $U_{\sigma j}(M)$. Therefore, taking pseudofinite cardinalities, we have $|U_{\sigma k}(M)| = |U_{\sigma j}(M)|$ for all $j \in \omega$. Let $n \in \omega$. Then $U_{\sigma 0}(M), \ldots, U_{\sigma n}(M)$ are disjoint subsets of $U_\sigma(M)$, whence $|U_\sigma(M)| \geq |U_{\sigma 0}(M)| + \ldots + |U_{\sigma n}(M)| = (n+1)|U_{\sigma k}(M)| > n |U_{\sigma k}(M)|$. Since this is true of every $n$, it follows that $\delta(U_\sigma(M)) > \delta(U_{\sigma k}(M))$.
\end{proof} 

\begin{proposition}
\label{defdisjoint}
Let $M = \prod_{i \to \mathcal{U}} M_i$ be a pseudofinite ultraproduct which satisfies $T$ and let $M^+$ be its counting pair extension. Then $M^+$ is parametrically definable in the disjoint union $M \sqcup \mathbb{R}^\star$.
\end{proposition}
\begin{proof}
We use Proposition \ref{defunion}. Let $\varphi(x,\bar{y})$ be an $L$-formula. By quantifier elimination, $\varphi(x,\bar{y})$ is equivalent to a disjunction of conjunctions of atomic formulas and their negations. In fact, it suffices to prove the statement in Proposition \ref{defunion} when $\varphi(x,\bar{y})$ is merely a single conjunction of atomic formulas (not negations), as pseudofinite cardinality satisfies inclusion-exclusion properties. For example, if we have defined $f_\varphi$ for each $\varphi$ a conjunction of formulas from $\{\phi, \psi, \theta\}$ , then we can define $f_{\phi \vee (\psi \wedge \neg \theta)}$ as $f_{\phi} + f_{\psi \wedge \neg \theta} - f_{\phi \wedge \psi \wedge \neg \theta} = f_\phi + (f_{\psi} - f_{\psi \wedge \theta}) - (f_{\phi \wedge \psi} - f_{\phi \wedge \psi \wedge \theta})$. 

So it suffices to consider a $\varphi(x,\bar{y})$ a conjunction of formulas of the form $U_\sigma(x)$ and $B_{\tau, \rho}(x,y_i)$. A conjunction of predicates $U_\sigma(x)$ is equivalent to a single $U_\sigma(x)$. If there is an instance of $B_{\tau, \rho}$ in the conjunction, then the defined set $\varphi(M,\bar{y})$ is either empty or a singleton, whence $(c_1,c_2) = (0,1)$ and $(\psi_1(\bar{y}), \psi_2(\bar{y})) = (\neg \exists \bar{x} \varphi(x,\bar{y}), \exists \bar{x} \varphi(x,\bar{y}))$ give and define the cardinalities of $\varphi(x,\bar{y})$. Otherwise $\varphi(x,\bar{y})$ is equivalent to the single predicate $U_\sigma(x)$; letting $c = |U_\sigma(M)|$, we get that $(c)$ and $(\bar{y} = \bar{y})$ give and define the cardinalities of $\varphi(x,\bar{y})$.
\end{proof}

\color{black}

\section{Simple does not imply (A)}

The theory in this result is inspired by Example 4.1.3 in \cite{psas}, which is the theory of a single equivalence relation with infinitely many equivalence classes. As explained in the introductory section, the condition $(A)_\phi$ fails in this ultraproduct for the formula $\phi(x,y) :=$ ``$\neg x E y$", although the theory of the ultraproduct is stable. In this way they show that direct converse to their result ``$(A)$ implies a simple and low theory'' is false.

However, there are many ways to satisfy this theory in a product of finite structures -- any sequence of finite equivalence relations such that both the number of classes and the size of the smallest class go to infinity will work. In particular, we can consider $\prod_{n \to \mathcal{U}} M_n$ where $M_n$ has $n$ equivalence classes, each of size $n$. In this case, the dimension operator $\delta$ is extremely well-behaved, and the condition (A) is satisfied (and the stronger condition (SA) is as well). Thus it still seemed possible to give a weaker potential converse, that given a simple pseudofinite theory $T$ one can find some pseudofinite ultraproduct $M \models T$ such that (A) is satisfied in $M$.

The example in this section, which was suggested by Alex Kruckman and Cameron Donnay Hill in conversation with the author, shows that this weakened converse is still false. Our theory is an expansion of the theory of infinite equivalence relation with infinitely many infinite classes, with additional functions which \emph{force} the cardinalities of equivalence classes to grow at a rate fast enough to have (A) fail at the formula ``$\neg x E y$'' as above. We do this by introducing a pairing function which bijects the largest equivalence class with the set of \emph{ordered pairs} from the second largest, and bijects that class with the set of ordered pairs from the third largest, and so on. We realize this pairing function as a pair of unary functions $f$ and $g$, so that $x \mapsto (f(x), g(x))$ is our bijection.

We give our construction in detail. Our language $L$ contains:
\begin{itemize}
\item A binary relation $E$, and
\item Two unary functions $f$ and $g$.
\end{itemize}

We frequently refer to $L$-terms as $\sigma(x)$, where $\sigma$ is a string in $\{f,g\}^{<\omega}$.\\

The theory $T_0$ says:
\begin{enumerate}[label=(\alph*)]
\item $E$ is an equivalence relation;
\item $f$ and $g$ are well-defined on $E$-equivalence classes, and project down to the same function on $M / E$. That is, if $x E y$ then $f(x) E g(y)$ (which implies $f(x) E f(y)$ and $g(x) E g(y)$);
\item There is a single $E$-class $C_{fin}$ such that $C_{fin} = \{x : f(x) = x\} = \{x : g(x) = x\}$;
\item There is a single $E$-class $C_{init}$ such that $C_{init} = \{x : f^{-1}(x) = \emptyset\} = \{x : g^{-1}(x) = \emptyset\}$;
\item If $x,y \notin C_{fin}$ and $f(x) E f(y)$ then $x E y$ (that is, $f,g$ are injective on equivalence classes except for $C_{fin}$ and $f^{-1}(C_{fin})$).
\item For each equivalence class $C$ which is not $C_{init}$ and each pair $(x, y) \in C^2$ there is a unique $z$ (if $C = C_{fin}$ then there is a unique $z \notin C_{fin}$) such that $f(z) = x$ and $g(z) = y$.

$T_0$ is finitely axiomatizable and has finite models (see Proposition \ref{eqpairpseudo}).

Our theory $T$ is the theory $T_0$ together with axiom schemata for
\item $\neg (f^k(x) E x)$ for each $k \in \omega$;
\item $E$ has infinitely many classes;
\item The equivalence classes of $E$ are infinite
\end{enumerate}

Later on, the following abbreviation will be helpful.

\begin{definition}
Let $M \models T$ and let $k \in \omega$. Then $C_{init+k}$ denotes the equivalence class $f^k(C_{init})$, and $C_{fin-k}$ denotes the equivalence class $f^{-k}(C_{fin}) \setminus f^{-(k-1)}(C_{fin})$.
\end{definition}

Clearly each $C_{init+k}$ and $C_{fin-k}$ is $\emptyset$-definable. When we prove quantifier elimination for $T$, we will add unary predicates for each $C_{init+k}$ to our language, since each of these classes is not otherwise quantifier-free definable.

Here we describe a method for constructing models of $T$. First let us define a simpler language and theory, which we will also use in proving quantifier elimination for $T$.

\begin{definition}
\label{tstar}
The language $L^\star$ has a unary function $S$ and two constant symbols $c_{init}, c_{fin}$.

The theory $T^\star$ says
\begin{itemize}
\item If $x \neq c_{init}$ then there is a $y$ with $S(y) = x$. If $x \neq c_{fin}$ then this $y$ is unique and $y \neq x$; if $x = c_{fin}$ then $S(x) = x$ and there is a unique $y \neq x$ with $S(y) = x$. There is no $y$ with $S(y) = c_{init}$.

\item If $x \neq c_{fin}$ then $S^n(x) \neq x$ (for all $n$)
\end{itemize}
\end{definition}

One class of models of $T^\star$ is \[(\mathcal{M}: S; c_{init}; c_{fin}) = ([a,b] : x \mapsto \min (succ(x),b); a; b\},\] where $[a,b]$ is an infinite closed interval of a discrete linear order, and $succ(x)$ is the successor function. $T^\star$ is strongly minimal and is known to have quantifier elimination (a slight variation on, for example, Exercise 3.4.3 in \cite{marker}). We shall use the quantifier elimination of $T^\star$ to give a quantifier elimination for $T$. If $M$ is an $L$-structure which satisfies $T$, we can define an $M^\star$ in $M^{eq}$ on the collection of equivalence classes $M / E$, where $c_{init}$ is $C_{init}$, $c_{fin}$ is $C_{fin}$, and $S([a]_E) = [f(a)]_E = [g(a)]_E$. Then $M^\star$ is a model of $T^\star$.

Let $I$ be a model of $T^\star$, and let $X$ be an infinite set and let $\pi : X \to X \times X$ be a bijection of $X$ with its own square. We define the $L$-structure $M_{I,\pi}$ as follows. Let the universe of $M_{I,\pi}$ be $I \times X$. Let $C_{init}(M)$ be $\{c_{init}\} \times X$ and let $C_{fin}(M)$ be $\{c_{fin}\} \times X$. Let $(i,x) E (i',x')$ exactly when $i = i'$. If $i = c_{fin}$ let $f((i,x)) = g((i,x)) = (i,x)$. If $i \neq c_{fin}$, let $f((i,x)) = (S(i), \pi_1(x))$ and $g((i,x)) = (S(i), \pi_2(x))$, where $\pi_1, \pi_2 : X \to X$ are the unary functions such that $\pi(x) = (\pi_1(x),\pi_2(x))$. The resulting structure is a model of $T$.

We observe, tangentially, that if $\pi, \pi'$  are different pairing functions on $X$, the structures $M_{I,\pi}$ and $M_{I,\pi'}$ may not be isomorphic. Let $X = \omega$. Suppose there is an $a \in X$ such that $\pi(a) = (a,a)$. Then $\sigma((a,c_{init})) = \tau((a,c_{init}))$ for all $\sigma, \tau \in \{f,g\}^{<\omega}$ of the same length. Suppose $\pi'$ has the property that whenever $\pi'(x) = (y,y)$ (so that $\pi_1'(x) = \pi_2'(x)$) we must have $y < x$. Then the type $\{C_{init}(x)\} \cup \{\sigma(x) = \tau(x) : \sigma, \tau \in \{f,g\}^{<\omega}$ are the same length$\}$ is not realized in $M_{I,\pi'}$ (for otherwise there is $a \in X$ with $\pi(a) = (b,b)$, and $\pi(b) = (c,c)$, and so on, giving a descending sequence $a > b > c > \ldots$), and so $M_{I,\pi}$ and $M_{I,\pi'}$ are not isomorphic.

\begin{proposition}
\label{eqpairpseudo}
1. $T$ is pseudofinite.

2.  Let $M = \prod_{i \to \mathcal{U}} M_i$ be a pseudofinite ultraproduct which satisfies $T$. Let $a_0,a_1,\ldots \in M$ such that $a_i \in f^{i}(C_{init})$.
 Then $\delta(\neg(x E a_0)) > \delta(\neg(x E a_0) \wedge \neg(x E a_1)) > \ldots$. In particular, $M$ fails to satisfy condition (A$)_\phi$ for the formula $\phi(x,y) =$``$\neg x E y$", and therefore $M$ fails to satisfy condtion (A).
\end{proposition}
\begin{proof}
1. For every pair of numbers $n,m$, there is a unique (up to isomorphism) finite structure $M_{n,m} \models T_0$ which has $n$ equivalence classes such that $|C_{fin}(M_{n,m})| = m$. In that case, $|M_{n,m}| = \sum_{i=0}^{n-1} m^{2^i}$. Each instance of the axiom schema for $T$ is clearly satisfied in $M_{n,m}$ for sufficiently large $n,m$. Therefore $T$ is pseudofinite. If $\mathcal{U}$ is an ultrafilter on $\omega \times \omega$ then $\prod_{(n,m) \to \mathcal{U}} M_{n,m} \models T_0$ always, and $\prod_{(n,m) \to \mathcal{U}} M_{n,m} \models T$ iff $\{(n,m) : n \geq j$ and $m \geq k\}$ is $\mathcal{U}$-big for all $j,k \in \omega$.

2. For each $k$, let $\varphi_k(x)$ be the formula $\bigwedge_{i=0}^k x \notin f^{i}(C_{init})$. 

For $k,m \in \omega$ let $D_{k,m}$ be the number $\sum_{i=0}^k m^{2^i}$. Then for $n > k$ we have that $|\varphi_k(M_{n,m})| = D_{(n-k-2),m}$, as $|M_{n,m}|$ is $D_{n-1,m}$ and $f^{j}(C_{init}(M_{n,m}))$ has size $m^{2^{(n-1)-j}}$ for $0 \leq j \leq k$.

If $\epsilon > 0$ is a positive real and $j < k$ then $(1 - \epsilon) m^{2^k-2^j} D_{j,m} < D_{k,m}$ for sufficiently large $m$, as both sides of this inequality are degree-$2^k$ polynomials in $m$. Therefore if $N \in \omega$ then $N \cdot \varphi_k(M_{n,m}) < \varphi_j(M_{n,m})$ for sufficiently large $m$ -- that is, for finite models of $T_0$ with $|C_{init}| > c_{N,j,k}$ for some $c_{N,j,k} \in \omega$. So if $(M_\lambda : \lambda \in \Lambda)$ is a family of finite models of $T_0$ and $\mathcal{U}$ is an ultrafilter on $\Lambda$ such that $M := \prod_{i \to \mathcal{U}} M_\lambda \models T$, then $\{\lambda \in \Lambda : |C_{init}(M_\lambda)| > c_{N,j,k}\} \in \mathcal{U}$, and so $\{\lambda \in \Lambda : N \cdot |\varphi_k(M_\lambda)| \leq |\varphi_j(M_\lambda)|\} \in \mathcal{U}$. Since this is true for every $N$, we obtain that $\delta(\varphi_k(M_\lambda)) < \delta(\varphi_j(M_\lambda))$. If $a_0,a_1,\ldots \in M$ are such that $a_i \in f^i(C_{init}(M))$ for each $i$, then $\varphi_k(x)$ is equivalent to $\bigwedge_{i=0}^k \neg x E a_i$, which proves the failure of $(A)_\phi$ for the formula $\phi(x,y) := \neg x E y$.
\end{proof}

The following lemma is the core property of $T$ from which we obtain quantifier elimination, as well as a more general result about the cardinalities of definable sets. Intuitively, it says that we may ``coordinatize'' an element of an equivalence class $C$ by freely choosing $2^n$ elements of $f^n(C)$, generalizing axiom (f) of $T_0$.

\begin{lemma}
\label{pairlem}
Let $\{s_1,\ldots,s_{2^n}\}$ be a complete enumeration of the strings in $\{f,g\}^n$. Let $b_1,\ldots,b_{2^n} \in C$ for some equivalence class $C \notin \{C_{init}, \ldots, C_{init+n-1}\}$. Then if $C \neq C_{fin}$, there is a unique $a \in M$ such that $s_i(a) = b_i$ for each $i$ (clearly, $a \in f^{-n}(C)$). If $C = C_{fin}$, there is a unique solution $a \in C_{fin-k}$ for each $k$ such that $b_i = b_j$ whenever $s_i$ and $s_j$ agree on their final $k$ bits (in particular, there is a unique solution in $C_{fin-n}$), and these are the only solutions in $M$.
\end{lemma}
\begin{proof}
We prove this by induction on $n$. For $n = 1$, this is Axiom (f) for $C \neq C_{fin}$, and follows from Axioms (f) and (c) for $C = C_{fin}$. Suppose the claim has been shown for $n$. Let us enumerate the strings in $\{f,g\}^{n+1}$ as $\{f s_1, fs_2,\ldots,f s_n, g s_1, g s_2, \ldots, g s_n\}$, where $\{s_1,\ldots,s_{2^n}\}$ is $\{f,g\}^n$. Let \[b_1,\ldots,b_{2^n},b_{2^n+1},\ldots,b_{2^{n+1}} \in C.\] 
 
If $C \neq C_{fin}$ then by induction, there are unique elements $c_1,c_2 \in M$ such that $s_i(c_1)= b_i$ and $s_i(c_2) = b_{2^n + i}$ for $i \leq 2^n$, and there is a unique $a \in M$ such that $f(a) = c_1$ and $g(a) = c_2$. It follows that $a$ is the unique element of $M$ such that $f s_i(a) = b_i$ and $g s_i(a) = b_{2^n + i}$ for each $1 \leq i \leq 2^n$. 

Suppose $C = C_{fin}$. If $s_i(x) \in C_{fin}$, we must have $x \in C_{fin-k}$ for some $k \leq n$. If $x \in C_{fin - k}$ then $s(x) = \tau(x)$ where $\tau$ is the final segment of $s$ of length $k$. It follows that if $s_i, s_j \in \{f,g\}^n$ agree on their final $k$ bits and $b_i \neq b_j$, then there is no solution to $s_i(x) = b_i \wedge s_j(x) = b_j$ in $C_{fin-k}$. On the other hand, suppose that $b_i = b_j = b_\tau$ whenever $s_i$ and $s_j$ both have $\tau \in \{f,g\}^k$ as a final segment. Then for $x \in C_{fin-k}$, the formula $\bigwedge_{i=1}^{2^n} s_i(x) = b_i$ is equivalent to $\bigwedge_{\tau \in \{f,g\}^k} \tau(x) = b_\tau$, which has a unique solution in $C_{fin-k}$.

\end{proof}

We note that Lemma \ref{pairlem} implies that for each $k$, we have $a = b$ if and only if $s(a) = s(b)$ for all $s \in \{f,g\}^k$ and $a E b$. This is used for syntactical manipulation of formulas (Lemma \ref{intermedform}) in our proof of quantifier elimination.

The stability of $T$ is an easy consequence of quantifier elimination in $T$ (Proposition \ref{qe}). The proof of quantifier elimination is rather involved, so we defer it until the end of the section.

\begin{proposition}
\label{stable}
$T$ is stable.
\end{proposition}
\begin{proof}
Let $M$ be a model of $T$, and let $A \subset M$ be a subset. We will show that all 1-types over $A$ are definable over $A$. Let $p \in S_1(A)$. By quantifier elimination in the language $L \cup \{C_{init},C_{init+1},\ldots\}$ (Proposition \ref{qe}), it suffices to show that the sets \[X_{\sigma,\tau} = \{a \in A : \mbox{``}\sigma(x) E \tau(a)\mbox{"} \in p\}\] and \[Y_{\sigma,\tau} = \{a \in A : \mbox{``}\sigma(x) = \tau(a)\mbox{"} \in p\}\] are definable over $A$ for all $\sigma, \tau \in \{f,g\}^{<\omega}$. The sets $Y_{\sigma, \tau}$ are either empty or of the form $\{a \in A : \tau(a) = \tau(a')\}$ for some fixed $a' \in A$, so they are definable over $A$. Similarly, the sets $X_{\sigma,\tau}$, if nonempty, are of the form $\{a \in A : \tau(a) E \tau(a')\}$ for some fixed $a' \in A$, which are definable over $A$ as well. Therefore all 1-types over $A$ are definable over $A$, which is one of the many equivalent conditions for the stability of $T$.
\end{proof}

By Proposition \ref{stable} (using the fact every stable theory is simple and low, see \cite[Remark 2.2]{buech}) and Proposition \ref{eqpairpseudo}, $T$ is a simple and low pseudofinite theory for which every pseudofinite ultraproduct satisfying $T$ fails to satisfy the GMS condition (A).

The rest of this section is spent proving quantifier elimination in $T$.

First we must expand our language $L$, adding unary predicates for the equivalence classes $C_{init}, C_{init+1},\ldots$ and $C_{fin}, C_{fin-1},\ldots$. These equivalence classes are $\emptyset$-definable in the language $\{f,g,E,=\}$, but the classes $C_{init+k}$ are not definable by quantifer-free formulas: the formula ``$C_{init}(x)$", for example, can be defined by the formula ``$\neg \exists z f(z) = x$". The final classes $C_{fin-k}$ are quantifier-free definable, but we include them anyway to make notation easier.

An outline of our argument is as follows. First, we define the property ``$\varphi$ has definable polynomial cardinality in $\theta$ over $R$", where $\varphi, \theta$ are formulas and $R$ is a subring of $\mathbb{R}^\star$. In Lemma \ref{standard}, we isolate some additional properties which, together with definable polynomial cardinalities, may be used to show quantifier elimination. 

We then define a class of formulas (the ``basic" formulas) and show that the hypotheses of Lemma \ref{standard} hold of all basic formulas (Lemma \ref{basic}). We use this to obtain this result for the wider class of ``intermediate" formulas (Lemma \ref{intermediate}), and then finally for certain conjunctions of atomic formulas and their negations (Lemma \ref{literals}), from which quantifier elimination follows (Proposition \ref{qe}).

Here we define what we mean by ``definable polynomial cardinality".

\begin{definition}
Let $(M,\mathbb{R}^\star)$ be a counting pair. Let $\varphi(\bar{x},\bar{y})$ and $\theta(\bar{z},\bar{y})$ be $L$-formulas, and let $R$ be a subring of $\mathbb{R}^\star$. Then we say $\varphi$ \emph{has (quantifier-free) definable polynomial cardinality in $\theta$ over $R$} if there are nonzero polynomials $F_1(X),\ldots,F_r(X)$ with coefficients from $R$ and (quantifier-free) formulas $\pi_1(\bar{y}),\ldots,\pi_r(\bar{y})$ so that

\begin{itemize}
\item If $M \models \pi_i(\bar{b})$ then $|\varphi(M^{|\bar{x}|},\bar{b})| = F_i(|\theta(M^{|\bar{z}|},\bar{b})|)$

\item If $M \models \bigwedge_i \neg \pi_i(\bar{b})$ then $\varphi(M^{|\bar{x}|},\bar{b}) = \emptyset$.
\end{itemize}
\end{definition}

Under certain additional assumptions, q.f.-definable polynomial cardinality can be used to obtain quantifier elimination results, via the following lemma.

\begin{lemma}
\label{standard}
Let $(M,\mathbb{R}^\star)$ be a counting pair. Suppose $\varphi$ has definable polynomial cardinality in $\theta$ over $\mathbb{R}$ (standard reals), witnessed by $F_1(X),\ldots,F_r(X)$ and $\pi_1(\bar{y}),\ldots,\pi_r(\bar{y})$. Suppose that whenever $M \models \pi_i(\bar{b})$ the set $\theta(M^{|\bar{z}|},\bar{b})$ is infinite. Then the formula $\exists x \varphi(\bar{x},\bar{y})$ is equivalent in $M$ to the formula $\bigvee_i \pi_i(\bar{y})$.
\end{lemma}

\begin{proof}
By DeMorgan's law, if $M \models \neg \bigvee_i \pi_i(\bar{b})$ then $M \models \neg \exists \bar{x} \varphi(\bar{x},\bar{b})$.

If $M \models \pi_i(\bar{b})$ for some $i$, then by assumption $|\theta(M^{|\bar{z}|},\bar{b})|$ is an infinite hyperinteger. The polynomial $F_i$ is nonzero and has standard coefficients, so therefore $F_i(|\theta(M^{|\bar{z}|},\bar{b})|)$ is nonzero -- for no nonzero polynomial with standard coefficients has an infinite hyperinteger as a root, since all roots in $\mathbb{R}$ and therefore in $\mathbb{R}^\star$ are bounded by a finite integer. Therefore $|\varphi(M^{|\bar{x}|},\bar{b})| \neq 0$, whence $M \models \exists \bar{x} \varphi(\bar{x},\bar{b})$.
\end{proof}

Next we define our family of \emph{basic} formulas.

\begin{definition}
Let $k$ be an integer. A \emph{$k$-basic} formula $\phi(x,y_0,\ldots,y_{m-1})$ is a formula of the form $\bigwedge_{i = 0}^{m-1} s_i(x) = y_i$, where $(s_0,\ldots,s_{m-1})$ is an $m$-tuple of strings $s_i \in \{f,g\}^k$.
\end{definition}

In the following lemma, we show that basic formulas have quantifier-free definable polynomial cardinality over a formula which satisfies the conditions of Lemma \ref{standard}.

\begin{lemma}
\label{basic}
Let $\phi(x,\bar{y})$ be a $k$-basic formula. Then $\phi$ has q.f.-definable polynomial cardinality in the formula ``$z E y_0$" over $\mathbb{Z}$.
\end{lemma}
\begin{proof}
Let $\phi(x,\bar{y})$ be $\bigwedge_{i=0}^{m-1} s_i(x) = y_i$, where each $s_i$ is a string in $\{f,g\}^k$.

 If $M \models \phi(a,\bar{b})$, then the following must hold:

\begin{itemize}

\item $b_i E b_j$ for all $i,j < m$ (since each string $s$ has the same length),

\item $b_i = b_j$ whenever $s_i = s_j$, and

\item $b_i \notin C_{init+l}$ for $i < m$ and $l < k$.

\end{itemize}

Let $\Psi(\bar{y})$ be the conjunction of the above conditions. 

Suppose $M \models \Psi(\bar{b})$. Let $S$ be the set of strings $\{s_i : i < m\}$. Suppose first that $b_0 \notin C_{fin}$ (and so $b_i \notin C_{fin}$, since all $b_i$ are $E$-equivalent as per $\Psi$). Then $\phi(M,\bar{b})$ is in definable bijection with $[b_0]_E^{2^k - |S|}$. To see this, we consider elements of this latter set as tuples indexed by strings in $\{f,g\}^k \setminus S$, and we pair $(c_s)_{s \notin S}$ with the unique $a$ such that $s_i(a) = b_i$ for $i < m$ and $s(a) = c_s$ for $s \notin S$ -- such an $a$ exists and is unique by Lemma \ref{pairlem}. We note that in this case, $\phi(M,\bar{b})$ is a subset of $f^{-k}([b_0]_E)$.

If $b_0 \in C_{fin}$ then, as in the proof of Lemma \ref{pairlem}, the situation is slightly more complicated. We have $\phi(M,\bar{b}) \subset C_{fin} \cup \ldots \cup C_{fin-k}$.  The set $\phi(M,\bar{b})$ has nonempty intersection with $C_{fin - l}$ precisely when $\bar{b}$ satisfies the following condtition: if $s, s' \in S$ and the length-$l$ prefix of $s_i$ equals the length-$l$ prefix of $s_j$, then $b_i = b_j$. Let $\rho_l(\bar{y})$ be a quantifier-free formula expressing this condition on $\bar{b}$. When $M \models \rho_l(\bar{b})$, the intersection $\phi(M,\bar{b}) \cap C_{fin_l}$ is in definable bijection with $C_{fin-l}^{2^n - |S'|}$, where $S'$ is the set of all length-$k$ prefixes of elements of $S$. We note also that $C_{fin-l}$ is in definable bijection with $C_{fin}^{2^l}$ by the same argument used in this proof, and so $\phi(M,\bar{b}) \cap C_{fin-l}$ is in definable bijection with a power of $C_{fin}$.

Therefore $\phi(M,\bar{b})$ in definable bijection with a power of $[b_0]_E$ or (if $[b_0]_E$ is $C_{fin}$) with one of finitely many disjoint unions of powers of $C_{fin}$, all determined by quantifier-conditions on $\bar{b}$. Since definable bijections preserve pseudofinite cardinality, we obtain that $\phi$ has qf-definable polynomial cardinality in the formula ``$z E y_0$" over $\mathbb{Z}$. 

\end{proof}

A brief modification of the above proof gets us a similar result for a wider class of formulas.

\begin{lemma} 
\label{basicplus}
Let $\phi(x,\bar{y}) := \bigwedge_i s_i(x) = y_i$ be a $k$-basic formula and let $\psi(x)$ be a conjunction of formulas $s(x) = t(x)$ with $s,t \in \{f,g\}^k$. Then $\varphi(x,\bar{y}) := \phi(x,\bar{y}) \wedge \psi(x)$ has qf-definable polynomial cardinality in ``$z E y_0$" over $\mathbb{Z}$.
\end{lemma}
\begin{proof}
Without loss of generality, $\psi(x)$ is of the form $\bigwedge_{t \sim t'} t(x) \sim t'(x)$, where $\sim$ is some equivalence relation on $\{f,g\}^k$. Now, $\varphi(M,\bar{b})$ is nonempty if and only if $\phi(M,\bar{b})$ is nonempty and $\bar{b}$ additionally satisfies the property that $b_i = b_j$ whenever $s_i \sim s_j$. Let $[t_0]_\sim,\ldots,[t_{l-1}]_\sim$ be the equivalence classes, ordered so that for some $j \leq l$ we have $\{[s_i]_\sim : i < m\} = \{[t_j]_\sim, \ldots, [t_{l-1}]_\sim\}$. Assume that all elements of $\bar{b}$ are from the same equivalence class $C$ (or else $\phi(M,\bar{b})$ is empty). Then, as in Lemma \ref{basic}, the set $\varphi(M,\bar{b})$ is in definable bijection with $C^j$ if $C \neq C_{fin}$, as we may freely choose elements of $C$ for each equivalence class of strings $[t]_\sim$ which is not specified by a formula $s_i(x) = y_i$. If $C = C_{fin}$ then as before, $\varphi(M,\bar{b})$ is in definable bijection with one of finitely many disjoint unions of powers of $C$, depending on some quantifier-free condition on $\bar{b}$.
\end{proof}

\begin{definition}
A \emph{$k$-intermediate} formula $\phi(x,y_1,\ldots,y_m)$ is a conjunction of formulas of the form $s(x) = s'(x)$ and $s(x) = t(y_i)$, where $s, s'$ and $t$ are all strings in $\{f,g\}^{<\omega}$, with each string $s$ and $s'$ having length $k$. 
\end{definition}

In this definition, we do not require every $\bar{y}$ variable to occur in a conjunct, and we do not require that both types of conjunct appear. We do not consider the empty conjunct to be $k$-intermediate, but we note that it is equivalent to the $k$-intermediate formula ``$x = x$".

Equivalently, a $k$-intermediate formula is a formula $\varphi(x,\bar{y})$ of the form $\psi(x)$ or the form $\phi(x,t_1(y_{i_1}),\ldots,t_n(y_{i_n})) \wedge \psi(x)$, where $\phi(x,y_1,\ldots,y_n)$ is a $k$-basic fomrula and $\psi(x)$ is a conjunction of formulas $s(x) = s'(x)$, with each $s, s'$ being a string of length $k$. The fact that this is equivalent depends on the function symbols in $L$ being unary, so that every $L$-term is of the form $t(v)$ for some variable $v$ and some string $t \in \{f,g\}^{<\omega}$.

The way we prove that certain conjunctions of atomic formulas have definable polynomial cardinality is by splitting the conjunctions into a part in the language $\{f,g,=\}$, and a part in the language $L$ without equality. The latter we call \emph{equivalence formulas}.

\begin{definition}
\label{eqdef}
An \emph{equivalence formula} is an $L$-formula which does not use the equality symbol.
\end{definition}

Atomic equivalence formulas are $C_{init+i}(v), C_{fin-i}(v)$ and $s(v) E t(w)$" with a natural number $i \in \omega$, strings $s,t, \in \{f,g\}^{<\omega}$, and variables $v,w$. An easy induction on formula length shows that if $\phi(\bar{x})$ is an equivalence formula and $\bar{a}, \bar{a}' \in M^{|\bar{x}|}$ with $a_i E a'_i$ for each $i$, then $M \models \phi(\bar{a})$ if and only if $M \models \phi(\bar{a}')$. In fact, a stronger phenomenon occurs.

Recall the language $L^\star = \{c_{init},c_{fin},S)$ and $T^\star$ (Definition \ref{tstar}). Recall also the remarks after that definition, that we may define $M^\star$, a model of $T^\star$, on the set of equivalence classes $[x]_E$ by interpreting $c_{init}$ as $C_{init}$, $c_{fin}$ as $C_{fin}$, and letting $S([a]_E)$ be $[f(a)]_E$. 
\begin{lemma}
\label{bi-interp}
For every equivalence formula $\phi(\bar{x})$ there is an $L^\star$-formula $\phi^\star(\bar{x})$ so that for all $\bar{a} \in M^{n}$, we have $M \models \phi(a_0,\ldots,a_{n-1})$ if and only if $M^\star \models \phi^\star([a_0]_E,\ldots,[a_{n-1}]_E)$.
\end{lemma}
\begin{proof}
If $\phi(x)$ is $C_{init+k}(x_i)$, let $\phi^\star(x_i)$ be $S^k(c_{init})$.

If $\phi(x_i)$ is $C_{fin-k}(x_i)$, let $\phi^\star(x_i)$ be $S^k(x_i) = c_{fin} \wedge S^{k-1}(x_i) \neq c_{fin}$.

If $\phi(\bar{x})$ is $s(x_i) E t(x_j)$, let $\phi^\star(x_i)$ be $S^{|s|}(x_i) = S^{|t|}(x_j)$.

Boolean connectives and quantifiers pass up directly.
\end{proof}

As a consequence of this lemma, we obtain

\begin{lemma}
\label{qelight}
Let $\Phi(x,\bar{y})$ be a quantifier-free equivalence formula. Then the formula $\exists x \Phi(x,\bar{y})$ is equivalent to a quantifier-free formula $\theta(\bar{y})$.
\end{lemma}
\begin{proof}
Let $\Phi^\star(x,\bar{y})$ be the $L^\star$ formula as in Lemma \ref{bi-interp}. Since the theory $T^\star$ has quantifier elimination, the formula $\exists x \Phi^\star(x,\bar{y})$ is equivalent in $T^\star$ to a quantifier-free formula $\theta^\star(\bar{y})$. Via the interpretation of $M^\star$ in $M$, there is an $L$-formula $\theta(\bar{y})$ so that $M \models \theta(b_0,\ldots,b_{m-1})$ if and only if $M^\star \models \theta^\star([b_0]_E,\ldots,[b_{m-1}]_E)$. This  happens if and only $M^\star \models \exists x \Phi^\star(x,[b_0]_E,\ldots,[b_{m-1}]_E)$, and that happens if and only if $M \models \exists x \Phi(x,b_0,\ldots,b_{m-1})$.
\end{proof}

We now show that intermediate formulas satisfy the hypotheses of Lemma \ref{standard}. We also demonstrate that sets definable by intermediate formulas intersect equivalence classes in a definable way, which will be used when we incorporate equivalence formulas into our analysis of formulas with definable polynomial cardinality, as per the remarks before Definition \ref{eqdef}.

\begin{lemma}
\label{intermediate}
Let $\phi(x,\bar{y})$ be a $k$-intermediate formula with a conjunct of the form $s(x) = t(y_i)$. Then $\phi$ has q.f.-definable polynomial cardinality in the formula ``$z E t(y_i)$".

Furthermore, $\phi(x,\bar{b})$ is a subset of $f^{-k}([t(b_i)]_E)$, and there are quantifier-free formulas $\eta_1(\bar{y}), \ldots, \eta_p(\bar{y})$ in the language $\{f,g,=\}$ and numbers $j_1, \ldots, j_p \in \{0,\ldots,k\}$ so that 

\begin{itemize}
\item If $M \models \exists x \phi(x,\bar{b})$ and $t(b_i) \in C_{fin}$ then $M \models \eta_i(\bar{b})$ for some $i$, and $\phi(M,\bar{b})$ intersects exactly the equivalence classes $C_{fin-k}, \ldots, C_{fin-j_i}$.
\end{itemize}
\end{lemma}

\begin{proof}
The formula $\phi(x,\bar{y})$ is equal to $\varphi(x,t_0(y_{i_0}),\ldots,t_{l-1}(y_{i_{l-1}})) \wedge \psi(x)$, where $\varphi$ is $k$-basic and $\psi$ is $k$-intermediate in the single variable $x$. We may arrange the variables in $\varphi$ so that $t_0(y_{i_0})$ is the term $t(y_i)$ assumed in the hypotheses of the lemma. Then the first statement of the lemma follows directly from Lemma \ref{basicplus}.

For the second statement of the lemma, we note that if $a \in \phi(M,\bar{b})$, then $M \models s(a) = t(b_i)$ and so $a \in f^{-|s|}([t(b_i)]_E)$. The rest of the lemma follows directly from the proof of Lemma \ref{basic}.
\end{proof}

In order to turn Lemma \ref{intermediate} into our general quantifier-elimination argument, we use the following syntactical lemma.

\begin{lemma}
\label{intermedform}
Let $\phi(x,\bar{y})$ be a conjunction of atomic $L$-formulas, each of which contains the variable $x$. Let $k$ be any number larger than the length of the largest $\{f,g\}$-string appearing in a term in $\phi$. Then $\phi(x,\bar{y})$ is equivalent to $\rho(x,\bar{y}) \wedge \eta(x,\bar{y})$, where $\rho$ is a $k$-intermediate formula and $\eta(x,\bar{y})$ is a quantifier-free equivalence formula.
\end{lemma}
\begin{proof}
We note that a conjunction of $k$-intermediate formulas is again a $k$-intermediate formula, and a conjunction of equivalence formulas is clearly another equivalence formula. Therefore it suffices to prove the lemma for a single atomic formula $\phi(x,y)$.

If $\phi(x,y)$ has no equality symbol in it, we are done (our $k$-intermediate formula may be ``$f^k(x) = f^k(x)$".

If $\phi(x,y)$ is $s(x) = t(y)$, let $k > \max\{ |s|, |t|\}$. Then by Lemma \ref{pairlem}, $\phi(x,y)$ is equivalent to \[s(x) E t(y) \wedge \bigwedge_{u \in \{f,g\}^{k - |s|}} us(x) = ut(y).\]

Suppose $\phi(x,y)$ is $s(x) = t(x)$, and again let $k > \max\{|s|,|t|\}$. If $|s| = |t|$ then as in the case of $s(x) = t(y)$, the formula $s(x) = t(x)$ is equivalent to \[\bigwedge_{u \in \{f,g\}^{k - |s|}} us(x) = ut(x)\] (the missing conjunct ``$s(x) E t(x)$" is always true). If $|s| > |t|$, let $u_1$ be any string in $\{f,g\}^{k-|s|}$ and let $u_2$ be any string in $\{f,g\}^{k - |t|}$. Then ``$s(x) = t(x)$" is equivalent to ``$[u_1s(x) = u_2t(x)] \wedge [t(x) \in C_{fin}]$". To see this, we note that if $s(x) = t(x)$ then we must have $t(x) \in C_{fin}$, for if $s'$ is the length-$|t|$ final segment of $s$ then $t(x) E s'(x)$, and if $s'(x)$ is not an element of the final class $C_{fin}$ then $s(x)$ will be in a different equivalence class from $s'(x)$. Then we have $u_1s(x) = s(x)$ and $u_2t(x) = t(x)$, since our unary functions $f,g$ are the identity on $C_{fin}$. On the other hand, if $t(x) \in C_{fin}$ and $u_1s(x) = u_2t(x)$ then $s(x) \in C_{fin}$, and so $s(x) = u_1s(x) = u_2t(x) = t(x)$.
\end{proof}

\begin{lemma}
\label{atomic}
Let $\phi(x,\bar{y})$ be a conjunction of atomic formulas, with at least one conjunct of the form $s(x) = t(y_i)$ or $s(x) E t(y_i)$. Then $\phi$ has q.f.-definable polynomial cardinality in ``$z E t'(y_i)$" over $\mathbb{Z}$, for some $t' \in \{f,g\}^{<\omega}$.
\end{lemma}
\begin{proof}
Let us write $\phi(x,\bar{y})$ as $\rho(x,\bar{y}) \wedge \xi(x,\bar{y})$, where $\rho$ is in the language $\{f,g,=\}$ and $\xi$ is an equivalence formula. By Lemma \ref{intermedform}, there is a $k$ so that $\rho(x,\bar{y})$ is equivalent to the conjunction of a $k$-intermediate formula $\varphi(x,\bar{y})$ and another equivalence formula $\xi'(x,\bar{y})$. 

Inspecting the proof of that lemma, we see that if $\rho(x,\bar{y})$ contains a conjunct of the form $s(x) = t(y_i)$, then $\varphi$ contains a conjunct of the form $s'(x) = t'(y_i)$ for some strings $s', t'$ which are initial extensions of $s$ and $t$. Then by Lemma \ref{intermediate}, the formula $\varphi(x,\bar{y})$ has q.f.-definable polynomial cardinality in the formula ``$z E t'(y_i)$" over $\mathbb{Z}$, as witnessed by polynomials $F_1(X),\ldots,F_r(X)$ and quantifier-free formulas $\pi_1(\bar{y}),\ldots,\pi_r(\bar{y})$. Furthermore, there are quantifier-free formulas $\eta_1(\bar{y}),\ldots,\eta_p(\bar{y})$ so that ``$\bigvee_j \eta_j(\bar{y})$" is equivalent to \[\exists x \varphi(x,\bar{y}) \wedge [t'(y_i) \in C_{fin}]\] and if $M \models \eta_j(\bar{b})$ then solutions to $\varphi(x,\bar{b})$ lie exactly in some finite set of final classes $C_{fin-l}$. In sum, letting $\eta_0(\bar{y})$ be the formula ``$t'(y_i) \notin C_{fin}$", we obtain formulas $\sigma_0(w,y_i),\sigma_1(w),\ldots,\sigma_p(w)$ where $\sigma_0(w)$ is the formula ``$s'(w) E t'(y_i)$" and $\sigma_j(w)$ is a disjunction of formulas $C_{fin-l}(w)$ for $j > 0$ so that if $M \models \exists x \varphi(x,\bar{b})$ then $M \models \bigvee_{j=0}^p \eta_j(\bar{b})$, and if $M \models \eta_j(\bar{b})$ then $\varphi(M)$ intersects $[d]_E$ exactly when $M \models \sigma_j(d)$.

Now we can incorporate the equivalence formula $\xi'(x,\bar{y}) \wedge \xi(x,\bar{y})$. The precise details here are particularly tedious, so we sketch the argument. To find the cardinality of $\phi(M,\bar{b})$, we look at the set $\varphi(M,\bar{b})$. The formulas $\pi_j(\bar{y})$ define the cardinality of this set, and the formulas $\eta_j(\bar{y})$ determine which equivalence classes this set intersects. If $M \models \eta_j(\bar{y})$, then there is a solution to $\varphi(x,\bar{b}) \wedge \xi'(x,\bar{b}) \wedge \xi(x,\bar{b})$ precisely if there is a solution to the equivalence formula $\sigma_j(w,\bar{b}) \wedge \xi(w,\bar{b}) \wedge \xi'(w,\bar{b})$. By Lemma \ref{qelight}, this is equivalent to a quantifier-free condition $\tau(\bar{y})$ on $\bar{b}$.

If the solution set is nonempty, then the cardinality is exactly the cardinality of $\varphi(M,\bar{b})$ if this set is contained in a single equivalence class (as in the case where $t'(b_i) \notin C_{fin}$), or is otherwise some smaller cardinality which is still polynomial in $C_{fin} = [t'(b_i)]_E$, with quantifier-free defining formulas given by a conjunction of a $\pi_j(\bar{y})$, an $\eta_{j'}(\bar{y})$, and the quantifier-free formula $\tau(\bar{y})$.  
\end{proof}

\begin{lemma}
\label{literals}
Let $\phi(x,\bar{y})$ be a conjunction of literals, with at least one conjunct of the form $s(x) = t(y_i)$ or $s(x) E t(y_i)$. Then $\phi$ has q.f.-definable polynomial cardinality in ``$z E t(y_i)$" over $\mathbb{Z}$.
\end{lemma}
\begin{proof}

Let $\phi(x,\bar{y})$ be $\rho(x,\bar{y}) \wedge \bigwedge_i \neg \eta_i(x,\bar{y})$, where $\rho$ is a conjunction of atomic formulas, one of which is ``$s(x) = t(y_i)$" or ``$s(x) E t(y_i)$", and each $\eta_i$ is atomic (with the possibility that there is no conjunct of this form). Then $\phi$ is equivalent to $\rho(x,\bar{y}) \wedge \neg \bigvee_i \eta_i(x,\bar{y})$. Then for any $\bar{b}$, we have $|\phi(M,\bar{b})| = |\rho(M,\bar{b})| - |\rho(M,\bar{b}) \wedge \bigvee_i \eta_i(M,\bar{b})| = |\rho(M,\bar{b})| - |\bigvee_i [\rho(M,\bar{b}) \wedge \eta_i(M,\bar{b})]|$.

Now we appeal to the inclusion-exclusion principle from basic combinatorics. Given finite subsets $A,B$ of some set, the cardinality $|A \cup B|$ is $|A| + |B| - |A \cap B|$. More generally, given finite sets $A_1, \ldots, A_n$, the set $A_1 \cup \ldots \cup A_n$ has cardinality \[\sum_{k=1}^n (-1)^{k+1}(\sum_{1 \leq i_1 < \ldots < i_k \leq n} |A_{i_1} \cap \ldots \cap A_{i_k}|).\]

Therefore the cardinality $|\bigvee_i [\rho(M,\bar{b}) \wedge \eta_i(M,\bar{b})]|$ is a sum/difference of cardinalities of the form $|\rho(M,\bar{b}) \wedge \eta_{i_1}(M,\bar{b}) \wedge \ldots \wedge \eta_{i_k}(M,\bar{b})|$. By Lemma \ref{atomic}, each of these formulas has q.f.-definable polynomial cardinality in ``$z E t(y_i)$" over $\mathbb{Z}$. It follows that $\bigvee_i [\rho(M,\bar{b}) \wedge \eta_i(M,\bar{b})]$ and therefore $\phi(x,\bar{y})$ do as well.
\end{proof}

The final ingredient in our quantifier elimination proof is an observation about $\emptyset$-definable sets.

\begin{lemma}
\label{emptydef}
Let $\phi(x)$ be a quantifier-free formula in the single variable $x$. Then there is a quantifier-free formula $\theta(w)$ in the unary relational language $\{C_{init}, C_{init+1}, \ldots\} \cup \{C_{fin},C_{fin-1},\ldots\}$ so that for each equivalence class $[d]_E$, the set $\phi(M)$ intersects $[d]_E$ if and only if $M \models \theta(d)$.
\end{lemma}
\begin{proof}
Atomic formulas in the single variable $x$ come in the forms ``$C_{init+i}(x)$", ``$C_{fin-i}(x)$", ``$s(x) E t(x)$", and ``$s(x) = t(x)$".

When $s$ and $t$ are strings of the same length, the formula $s(x) E t(x)$ is always true. When $|s| > |t|$, $s(x) E t(x)$ happens if and only if $[s(x)]_E = [t(x)]_E = C_{fin}$, which is true if and only if $x \in \bigcup_{i \leq |t|} C_{fin-i}(x)$. So any Boolean combination of single-variable formulas in $x$ without the equality symbol is equivalent to a formula in the language $\{C_{init+i} : i \in \omega\} \cup \{C_{fin-i} : i \in \omega\}$. 

Suppose $k > \max( |s|, |t|)$. Then by Lemma 4.17, ``$s(x) = t(x)$" is equivalent to $\rho(x) \wedge \eta(x)$, where $\rho$ is $k$-intermediate and $\eta$ is a quantifier-free equivalence formula. By definition, $\rho$ is a conjunction of the form $\bigwedge_i s_i(x) = t_i(x)$, where $s_i, t_i \in \{f,g\}^k$ for each $i$.

It follows that there exists a $k$ such that $\phi(x)$ is equivalent to a Boolean combination of formulas of the form ``$C_{init+i}(x)$", ``$C_{fin-i}(x)$" and ``$s(x)  = t(x)$", where all $s,t \in \{f,g\}^k$. The statement of the lemma passes up through disjunctions, so we may assume $\phi(x)$ is a conjunction of such formulas and their negations. Let us write $\phi(x)$ as $\rho(x) \wedge \eta(x)$, where $\rho$ is in the language $\{f,g,=\}$ and $\eta$ is in the language $\{C_{init+i}(x) : i \in \omega\} \cup \{C_{fin-i}(x) : i \in \omega\}$. Without loss of generality, $\rho(x)$ is $\bigwedge_{s \sim t} s(x) = t(x) \wedge \bigwedge_j s'_j(x) \neq t'_j(x)$ where $\sim$ is an equivalence relation on $\{f,g\}^k$ and all $s'_j,t'_j$ are strings in $\{f,g\}^k$. Then $\rho(M)$ is nonempty if and only if $s'_j \not \sim t'_j$ for all $j$. If $\phi(M)$ is  nonempty and $C$ is not $C_{fin-i}$ for any $i < k$ then $\rho(C)$ is nonempty, and for $i < k$ the set $\rho(C_{fin-i})$ is nonempty if and only if $s'_j$ and $t'_j$ do not have the same length-$i$ final string for each $j$. In all, there is a quantifier-free formula $\pi(w)$ in the language $\{C_{fin-i}(w) : i \in \omega\}$ such that for all $d \in M$, $\rho([d]_E)$ is nonempty if and only if $M \models \pi(d)$. Then $\phi([d]_E)$ is nonempty if and only if $M \models \pi(d) \wedge \eta(d)$. This proves the lemma, with $\theta(w) = \pi(w) \wedge \eta(w)$.

\end{proof}

We now finish our proof of quantifier elimination.

\begin{proposition}
\label{qe}
$T$ has quantifier elimination in the expanded language \[\{f,g,E\} \cup \{C_{init}, C_{init+1}, C_{init+2}, \ldots\}.\]
\end{proposition}
\begin{proof}
We recall that we have been working in the expanded language \[\{f,g,E\} \cup \{C_{init}, C_{init+1},\ldots\} \cup \{C_{fin}, C_{fin-1},\ldots\}\] since the remarks after Proposition \ref{stable}. The predicates $C_{fin-k}$ are all quantifier-free definable in $\{f,g,E\}$; for example, ``$C_{fin}(x)$" is equivalent to ``$f(x) = x$", and we may define the other final predicates recursively, as ``$C_{fin-k}(x)$" holds precisely when ``$C_{fin-(k-1)}(f(x)) \wedge \neg C_{fin-(k-1)}(x)$" holds.

It suffices to show that $\exists x \Phi(x,\bar{y})$ is equivalent to a quantiifer-free formula $\eta(\bar{y})$ whenever $\Phi(x,\bar{y})$ is a conjunction of literals. We may assume that the variable $x$ appears in each literal. 

If there is a literal in $\Phi$ of the form $s(x) = t(y_i)$ or $s(x) E t(y_i)$ then we apply Lemma \ref{literals} and Lemma \ref{standard} to obtain quantifer-free formulas $\pi_1(\bar{y}),\ldots,\pi_r(\bar{y})$ so that in $M$, the formula $\exists x \Phi(x,\bar{y})$ is equivalent to the quantifier-free formula $\bigvee_i \pi_i(\bar{y})$.

If not, then let us write $\Phi$ as $\psi(x) \wedge \varphi(x,\bar{y}) \wedge \bigwedge_i s_i(x) \neq t_i(y_{j_i})$, where $\psi(x)$ is a conjunction of literals in the variable $x$ and $\varphi(x,\bar{y})$ is an equivalence formula. By Lemma \ref{emptydef}, there is a quantifier-free formula $\eta(w)$ in the language $\{C_{init}(w), C_{init+1}(w), \ldots\} \cup \{C_{fin}(w), C_{fin-1}(w), \ldots\}$ so that for all $d \in M$, the set $\psi(M)$ intersects $[d]_E$ if and only if $M \models \eta(d)$. Therefore $M \models \exists x [\psi(x) \wedge \varphi(x,\bar{y})$ if and only if $M \models \exists w [\eta(w) \wedge \varphi(w,\bar{y})]$. If we have $M \models \exists x [\psi(x) \wedge \varphi(x,\bar{y})]$ then $M \models \exists x \Phi(x,\bar{y})$, since we may choose a witness for $x$ in $\psi(x) \wedge \varphi(x,\bar{y})$ so that $s_i(x) \neq t_i(x_{j_i})$ for each $i$. So $M \models \exists x \Phi(x,\bar{y})$ if and only if $M \models \exists w [\eta(w) \wedge \varphi(w,\bar{y})]$. Since $\eta \wedge \varphi$ is an equivalence formula, Lemma \ref{qelight} gives us a quantifier-free equivalence formula $\theta(\bar{y})$ so that $\exists w \eta(w,\bar{y}) \wedge \varphi(w,\bar{y})$ is equivalent to $\theta(\bar{y})$. Therefore $\exists x \Phi(x,\bar{y})$ is equivalent to $\theta(\bar{y})$.

\end{proof}

\end{document}